\documentclass[11pt]{article}
\usepackage{amsmath}
\usepackage{amsfonts}
\usepackage{amssymb}
\usepackage{mathrsfs}
\usepackage{amsthm}
\usepackage{palatino}
\usepackage[margin=2.5cm, vmargin={1.5cm}]{geometry}

\usepackage{dcolumn}

\usepackage{times}
\usepackage{graphicx}
\usepackage{xcolor}

\usepackage{pstricks}
\usepackage{pst-plot}
\usepackage{pst-grad}

\usepackage{bm}

\renewcommand\footnotemark{}

\begin{document}

\title{Integrality of the higher Rademacher symbols}

\author{
  Cormac ~O'Sullivan\footnote{{\it Date:} July 15, 2023.
\newline \indent \ \ \
  {\it 2020 Mathematics Subject Classification:} 11F20, 11F67, 11B68, 11R11
\newline \indent \ \ \
{\em Key words and phrases.} Rademacher symbols, integrality, Bernoulli numbers, zeta function values.
  \newline \indent \ \ \
Support for this project was provided by a PSC-CUNY Award, jointly funded by The Professional Staff Congress and The City
\newline \indent \ \ \
University of New York.}
  }

\date{}

\maketitle


\def\s#1#2{\langle \,#1 , #2 \,\rangle}

\def\F{{\frak F}}
\def\C{{\mathbb C}}
\def\R{{\mathbb R}}
\def\Z{{\mathbb Z}}
\def\Q{{\mathbb Q}}
\def\N{{\mathbb N}}
\def\G{{\Gamma}}
\def\GH{{\G \backslash \H}}
\def\g{{\gamma}}
\def\L{{\Lambda}}
\def\ee{{\varepsilon}}
\def\K{{\mathcal K}}
\def\Re{\mathrm{Re}}
\def\Im{\mathrm{Im}}
\def\PSL{\mathrm{PSL}}
\def\SL{\mathrm{SL}}
\def\Vol{\operatorname{Vol}}
\def\lqs{\leqslant}
\def\gqs{\geqslant}
\def\sgn{\operatorname{sgn}}
\def\res{\operatornamewithlimits{Res}}
\def\li{\operatorname{Li_2}}
\def\lip{\operatorname{Li}'_2}
\def\pl{\operatorname{Li}}

\def\ei{\mathrm{Ei}}

\def\clp{\operatorname{Cl}'_2}
\def\clpp{\operatorname{Cl}''_2}
\def\farey{\mathscr F}

\def\dm{{\mathcal A}}
\def\ov{{\overline{p}}}
\def\ja{{K}}

\def\nb{{\mathcal B}}
\def\cc{{\mathcal C}}
\def\nd{{\mathcal D}}

\def\u{{\text{\rm u}}}
\def\v{{\text{\rm v}}}
\def\ib{{\text{\,\rm i}}}
\def\jb{{\text{\,\rm j}}}
\def\qq{{f}}

\newcommand{\stira}[2]{{\genfrac{[}{]}{0pt}{}{#1}{#2}}}
\newcommand{\stirb}[2]{{\genfrac{\{}{\}}{0pt}{}{#1}{#2}}}
\newcommand{\eu}[2]{{\left\langle\!\! \genfrac{\langle}{\rangle}{0pt}{}{#1}{#2}\!\!\right\rangle}}
\newcommand{\eud}[2]{{\big\langle\! \genfrac{\langle}{\rangle}{0pt}{}{#1}{#2}\!\big\rangle}}
\newcommand{\norm}[1]{\left\lVert #1 \right\rVert}
\newcommand{\dx}[1]{\overset{*}{#1}}

\newcommand{\e}{\eqref}
\newcommand{\bo}[1]{O\left( #1 \right)}
\newcommand{\ol}[1]{\,\overline{\!{#1}}} 


\newtheorem{theorem}{Theorem}[section]
\newtheorem{lemma}[theorem]{Lemma}
\newtheorem{prop}[theorem]{Proposition}
\newtheorem{conj}[theorem]{Conjecture}
\newtheorem{cor}[theorem]{Corollary}
\newtheorem{assume}[theorem]{Assumptions}
\newtheorem{adef}[theorem]{Definition}

\numberwithin{figure}{section}
\numberwithin{table}{section}


\newcounter{counrem}
\newtheorem{remark}[counrem]{Remark}

\renewcommand{\labelenumi}{(\roman{enumi})}
\newcommand{\spr}[2]{\sideset{}{_{#2}^{-1}}{\textstyle \prod}({#1})}
\newcommand{\spn}[2]{\sideset{}{_{#2}}{\textstyle \prod}({#1})}

\numberwithin{equation}{section}

\let\originalleft\left
\let\originalright\right
\renewcommand{\left}{\mathopen{}\mathclose\bgroup\originalleft}
\renewcommand{\right}{\aftergroup\egroup\originalright}

\bibliographystyle{alpha}

\begin{abstract}
Rademacher symbols may be defined in terms of Dedekind sums, and give the value at zero of the zeta function associated to a narrow ideal class of a real quadratic field. Duke extended these symbols to give the zeta function values at all negative integers. Here we prove Duke's  conjecture that these higher Rademacher symbols are integer valued, making the above zeta value denominators as simple as the corresponding Riemann zeta value denominators. The proof uses detailed properties of Bernoulli numbers, including a generalization of the Kummer congruences.
\end{abstract}

\section{Introduction}

For $\G= \SL(2,\Z)$, the Rademacher symbol $\Psi:\G \to \Z$ is not quite a homomorphism, but  possesses many important properties that link it to modular forms, continued fractions, knots and much more, as discussed in the introduction to the compelling paper \cite{duke}. The value at $s=0$ of the zeta function associated to a narrow ideal class of a real quadratic field may also be expressed in terms of $\Psi$. Duke formulated the higher Rademacher symbols $\Psi_n$ in \cite{duke} so that they provide the corresponding values at $s=1-n$ for $n=2, 3, 4, \dots$. This incorporates the work of Siegel in \cite{Si68}.

After setting up some notation, we give the definition of $\Psi_n$ next in terms of  general Dedekind sums. Let $T=\left(\begin{smallmatrix}1&1\\0&1\end{smallmatrix}\right)$, $V=\left(\begin{smallmatrix}1&0\\1&1\end{smallmatrix}\right)$ and for $A=\left(\begin{smallmatrix}a&b\\c&d\end{smallmatrix}\right)$ in  $\G$, set
\begin{equation*}
  \u(A):= \gcd(b,c,d-a)\gqs 1, \qquad \v(A):= \sgn(a+d)\frac{c}{\u(A)}.
\end{equation*}
The Bernoulli numbers $B_n$ play an important role from the start. Define $\ib_n$, $\jb_n$ by means of
\begin{equation} \label{ij}
   \frac{\ib_n}{\jb_n}=-\frac{B_{2n}}{2n}= \zeta(1-2n) \qquad \qquad (n \in \Z_{\gqs 1}),
\end{equation}
with $\ib_n/\jb_n$ in lowest terms and $\jb_n>0$.
For $m\gqs 0$,  the periodized Bernoulli polynomials $\ol{B}_m(x)$ are  $B_m(x-\lfloor x\rfloor)$, with the stipulation that $\ol{B}_1(x)=0$ for $x\in \Z$.
The general Dedekind sum, \cite{Ca54,LJC}, is
\begin{equation}\label{ded}
  S_{r,s}(a,c):=\sum_{h \bmod{|c|}} \ol{B}_r\left( \frac{ah}c\right) \ol{B}_s\left( \frac{h}c\right) \qquad \qquad (r,s \in \Z_{\gqs 0}, \ c \in \Z_{\neq 0}).
\end{equation}
Finally, for $r,s \in \Z_{\gqs 0}$, $r+s\gqs 2$, define the following hypergeometric polynomial  of degree at most $r+s-1$,
\begin{equation}\label{frs}
  F_{r,s}(z)  :=\frac{(r+s-2)!}{r! s!} {}_2 F_1\left( 1-r,1-s;\frac {3 -r-s}2;\frac{z+2}4 \right).
\end{equation}

\begin{adef} \cite[Sect. 2]{duke} {\rm
Let $n$ be a positive integer. For $A=\left(\begin{smallmatrix}a&b\\c&d\end{smallmatrix}\right)$ in $\G$ with $c \neq 0$,
\begin{equation} \label{sg}
  \Psi_n(A):= -\sgn(c)  \cdot  \v(A)^{n-1}  \cdot \jb_n \sum_{r+s=2n} F_{r,s}(a+d) \cdot S_{r,s}(a,c).
\end{equation}
If $n=1$ then $3\sgn(c(a+d))$ should be subtracted from the right of \e{sg}. When $c=0$ and hence $A=\pm T^k$, set $\Psi_n(\pm I):=0$ and $\Psi_n(\pm T^k):=\Psi_n(\pm V^{-k})$.} 
\end{adef}

Then $\Psi_n(A)$ may be computed for any $A$ as a finite sum of rational numbers. Here, $\Psi_1$ equals the original Rademacher symbol $\Psi$. Since $\Psi_n(-A)=\Psi_n(A)$, these  symbols are well-defined on the modular group $\G/\{\pm I\}$. The fundamental properties of $\Psi_n$ are given after \cite[Eq. (2.4)]{duke}. These include that $\Psi_n(A)$ does not change under conjugation, $A \mapsto BAB^{-1}$, and for $n\gqs 2$,  $\Psi_n$ sends elliptic elements to $0$ and parabolic elements to integer multiples of $\ib_n$.

Associated to each narrow ideal class $\mathcal A$ there is a hyperbolic $A \in \G$.  Duke gives an elegant direct proof in \cite[Thm. 4]{duke} of the connection
\begin{equation} \label{zet}
  \zeta(1-n, \mathcal A) = \frac{\Psi_n(A)}{\jb_n},
\end{equation}
between values of the corresponding zeta function and the higher Rademacher symbols. Equivalent identities were proved by Hecke for $n=1$ and by Siegel and Shintani in general.
Comparing the form of \e{zet} with \e{ij} and the Riemann zeta function value $\zeta(1-2n)$, it is natural to ask if $\Psi_n(A)$ is an integer. In 1976, Zagier gave a related formula for $\zeta(1-n, \mathcal A)$, but was unable to find a good estimate for the denominator. Zagier's formula is rederived in Corollary \ref{coo}. 

\begin{conj} \cite[Sect. 2]{duke} \label{con}
For every $n\gqs 2$, the higher Rademacher symbol satisfies $\Psi_n: \G \to \Z$.
\end{conj}

Our main result is that Duke's above conjecture is true; this is Corollary \ref{cos}. We prove a slightly stronger reformulation of Conjecture \ref{con}, given in \cite{duke} in terms of a certain polynomial $G_n(x)$ with coefficients involving Bernoulli numbers. This polynomial is described in the next section, where we use  reciprocity relations  to show that it is a period polynomial.  To prove that $G_n(x)$ has integer coefficients requires  further Bernoulli number properties, described in Section \ref{divb}, and including Cohen's extension of the classical Kummer congruences. As reviewed in Section \ref{ddu}, $G_n(x)$ having integer coefficients implies Conjecture \ref{con}.

For Rademacher symbols described in the wider context of general Fuchsian groups, see the overview in \cite{Bu}. The ideas in this paper should be applicable there too.

\vskip 3mm
{\bf Acknowledgements.} Many thanks to Karen Taylor for helpful conversations about this work and  Siegel's paper \cite{Si68}.

\section{Preliminaries}
\subsection{The spaces $\mathcal{P}_n$ and $\mathcal{W}_n$}
A positive integer $n$ is fixed in the background. Write $x=(x_1,x_2)$ and define
\begin{equation} \label{qrxd}
  Q_r(x):= \left[ z^r\right] \left( (z-x_1)(z-x_2)\right)^{n-1},
\end{equation}
where $[z^r]$ extracts the coefficient of $z^r$.
Then $Q_r(x)$ is a symmetric polynomial in $x_1$ and $x_2$ of homogeneous degree $r^*:=2n-2-r$ for $0\lqs r\lqs 2n-2$. Explicitly,
\begin{equation} \label{qrx}
  Q_r(x) = (-1)^r (x_1 x_2)^{n-1-r} \sum_{u+v=r} \binom{n-1}{u} \binom{n-1}{v} x_1^u x_2^v.
\end{equation}
Let $\mathcal{P}_n$ be the $\C$ vector space with basis $Q_r(x)$  for $0\lqs r\lqs 2n-2$.

The group $\G$ is generated by $ T=\left(\begin{smallmatrix}1&1\\0&1\end{smallmatrix}\right)$ and $S=\left(\begin{smallmatrix}0&-1\\1&0\end{smallmatrix}\right)$ with $S^2=-I$.
Also write
\begin{equation*}
  U=TS=\left(\begin{smallmatrix}1&-1\\1&0\end{smallmatrix}\right), \qquad V= TST = -ST^{-1}S = \left(\begin{smallmatrix}1&0\\1&1\end{smallmatrix}\right),
\end{equation*}
with $U^3=-I$. For $A=\left(\begin{smallmatrix}a&b\\c&d\end{smallmatrix}\right)$, its entries may be specified with $a=a_A$, $b=b_A$ and so on.
As described in \cite[Sect. 5]{duke}, there is a natural action of $\G$ on $f(x) \in \mathcal{P}_n$ with, for example,
\begin{equation} \label{rec}
  f|T(x)=f(x_1+1,x_2+1), \qquad f|S(x)=(x_1 x_2)^{n-1} f(-1/x_1,-1/x_2),
\end{equation}
for the generators. We have
\begin{equation} \label{rec2}
  Q_r|T(x)= \sum_{\ell=0}^{2n-2} (-1)^{\ell+r}\binom{\ell}{r} Q_\ell(x), \qquad Q_r|S(x)=(-1)^r Q_{r^*}(x).
\end{equation}
With $U=TS$, it follows that
\begin{equation} \label{qqq}
  Q_r|U(x)= (-1)^{r}\sum_{\ell=0}^{2n-2} \binom{\ell^*}{r} Q_\ell(x), \qquad Q_r|U^2(x)=(-1)^{r}\sum_{\ell=0}^{2n-2} \binom{\ell}{r^*} Q_\ell(x).
\end{equation}

Any $f$ in $\mathcal{P}_n$ may be written as
\begin{equation} \label{any}
  f(x)=  \sum_{r=0}^{2n-2} c_r Q_r(x).
\end{equation}
The next result follows directly from \e{rec2} and \e{qqq}. See \cite[p. 199]{KZ} for similar relations.

\begin{prop} \label{uu}
We have that $f$ in \e{any} satisfies $f|(1+S)=0$ if and only if $c_{r^*}+(-1)^r c_r=0$ for $0\lqs r\lqs 2n-2$.
We have $f|(1+U+U^2)=0$ if and only if
\begin{equation*}
  \sum_{r=0}^{2n-2} \left[ \delta_{r,\ell}+(-1)^r \binom{\ell^*}{r}+(-1)^r \binom{\ell}{r^*}\right]c_r = 0
\end{equation*}
for all $\ell$ with $0\lqs \ell \lqs 2n-2$. Here, $\delta_{r,\ell}$ is the Kronecker delta.
\end{prop}

Let $\mathcal{W}_n$ be the subspace of $\mathcal{P}_n$ consisting of those $f$s  satisfying $f|(1+S)=0$ and $f|(1+U+U^2)=0$; these are the two-variable period polynomials of \cite[Sect. 6]{duke}.

\subsection{The polynomial $G_n(x)$}
For integers $m\gqs 0$  define
\begin{equation} \label{eee}
  \beta_m:= \zeta(-m)=(-1)^m \frac{B_{m+1}}{m+1},
\end{equation}
(following the convention in the references that $B_1=-1/2$).
Set $\xi_0=\xi_{2n-2}=0$ and for $1\lqs r\lqs 2n-3$,
\begin{equation} \label{xi}
   \xi_r :=   \beta_{2n-1}\left(\beta_{r}+\beta_{r^*} + \frac{\delta_{r,1}+\delta_{r^*,1}}{2n-2}\right)-\beta_{r}\beta_{r^*} .
\end{equation}
Note that $\xi_{r^*}=\xi_r$ and also $\xi_r=0$ for all  even indices $r$. With \e{ij} and \e{qrxd}, define
\begin{equation} \label{grx}
  G_n(x) := \jb_n \sum_{r=0}^{2n-2} \xi_r \cdot Q_r(x)\qquad \qquad (n \in \Z_{\gqs 1}),
\end{equation}
and this is equivalent to \cite[Eq. (2.19)]{duke}. Each $G_n(x)$ is in $\Q[x_1,x_2]$ with degree at most $2n-3$ and only odd degree terms. These polynomials  are symmetric and, by the right identities in \e{rec}, \e{rec2}, reciprocal:
\begin{equation*}
  G_n(x_1,x_2)= G_n(x_2,x_1), \qquad (x_1 x_2)^{n-1} G_n(1/x_1,1/x_2) = G_n(x_1,x_2).
\end{equation*}
 The first $n$ for which $G_n$ is not identically $0$ is $n=6$:
\begin{multline*}
 G_6(x)= 72 (x_1^5 x_2^4+ x_1^4 x_2^5+x_1+ x_2)-75 (x_1^5 x_2^2+ x_1^2 x_2^5+x_1^3+ x_2^3)
 +6 (x_1^5+ x_2^5) \\-375(x_1^4 x_2^3+ x_1^3 x_2^4+x_1^2 x_2+ x_1 x_2^2)+150 (x_1^4 x_2
 + x_1 x_2^4)+600 (x_1^3 x_2^2+ x_1^2 x_2^3).
\end{multline*}

\begin{prop} \label{last}
For all $n\gqs 1$ we have $G_n(x)$ in $\mathcal{W}_n$.
\end{prop}

This important property of $G_n(x)$ is required for the proof of \e{dkm} below and stated in \cite[Lemma 13]{duke}. We may give a relatively simple proof based on the reciprocity relations for Bernoulli numbers  discussed in \cite{AD08}. These relations are interesting in their own right.

The first is due to Saalsch\"utz in 1892:
for all  $k$, $m$ in $\Z_{\gqs 0}$,
\begin{equation}\label{recip}
   \sum_j (-1)^j \binom{k}{j} \beta_{m+j} +  \sum_j (-1)^j \binom{m}{j} \beta_{k+j} = -\frac{ k! m!}{(k+m+1)!}.
\end{equation}
 Note that our $\beta_m$ equals $(-1)^{m}$ times the $\beta_m$ appearing in \cite{AD08}. Where the range of summation  is not specified, as in \e{recip},  we mean that it is  over the finite set of integers where the contained binomial coefficients are defined and non-zero.

The second reciprocity relation follows from a Bernoulli polynomial identity of Nielsen from 1923, as described in \cite{AD08}. For all  $k$, $m$ in $\Z_{\gqs 0}$,
\begin{multline}\label{recip2}
   \sum_j (-1)^j \binom{k}{j} \beta_{m+j}\beta_{k-j} +  \sum_j (-1)^j \binom{m}{j} \beta_{k+j} \beta_{m-j}\\
   = \beta_k \beta_m -\left(\frac{ k! m!}{(k+m+1)!} +\frac{(-1)^k}{k+1}+\frac{(-1)^m}{m+1}\right)\beta_{k+m+1}.
\end{multline}
Zagier also proved the underlying Bernoulli polynomial identity in the appendix 
to \cite{bern}. A generalization of \e{recip2} is provided in \cite[Thm. 1]{AD08}.

\begin{proof}[Proof of Proposition \ref{last}]
Since $G_1=0$ we may assume that $n\gqs 2$.  Proposition \ref{uu} implies that $G_n|(1+S)=0$ and that proving $G_n|(1+U+U^2)=0$ is equivalent to showing
\begin{equation}\label{try}
 \sum_{r=0}^{2n-2} \left[ (-1)^r \binom{\ell^*}{r}+(-1)^r \binom{\ell}{r^*}\right]\xi_r = -\xi_\ell,
\end{equation}
for all $\ell$ with $0\lqs \ell \lqs 2n-2$.
The left side of \e{try} has the form
\begin{equation*}
  -\frac{\beta_{2n-1}}{2n-2} S_1+ \beta_{2n-1} S_2 - S_3,
\end{equation*}
for
\begin{align*}
  S_1 & = \binom{\ell}{1}+ \binom{\ell}{1^*}+\binom{\ell^*}{1}+ \binom{\ell^*}{1^*},\\
  S_2 & = \sum_{r=1}^{2n-3} (-1)^r\left[  \binom{\ell}{r^*}\beta_{r^*} +\binom{\ell^*}{r}\beta_r+ \binom{\ell}{r^*}\beta_r + \binom{\ell^*}{r}\beta_{r^*}\right],\\
  S_3 & = \sum_{r=1}^{2n-3} (-1)^r\left[ \binom{\ell}{r^*}\beta_r \beta_{r^*}  + \binom{\ell^*}{r}\beta_r\beta_{r^*}\right].
\end{align*}

Starting with $S_2$, extending the sum to all $r$ gives
\begin{equation} \label{s2}
  S_2  = 1+\frac{\delta_{\ell,0}+\delta_{\ell^*,0}}2 +\sum_{r} (-1)^r\left[  \binom{\ell}{r^*}\beta_{r^*} +\binom{\ell^*}{r}\beta_r+ \binom{\ell}{r^*}\beta_r + \binom{\ell^*}{r}\beta_{r^*}\right],
\end{equation}
since $\beta_{2n-2}=0$. Taking $m=0$ in
\e{recip} means
\begin{equation*}
  \sum_j (-1)^j \binom{k}{j} \beta_{j} = -\beta_{k}  -\frac{1}{k+1},
\end{equation*}
giving the first two sums in \e{s2}. Next, with $k=\ell$ and $m=\ell^*$ in
\e{recip} we find, after replacing the first index $j$ by $\ell-r$ and the second $j$ by $\ell^*-r$,
\begin{equation} \label{sim}
  \sum_{r} (-1)^r\left[   \binom{\ell}{r^*}\beta_r + \binom{\ell^*}{r}\beta_{r^*}\right]= -\frac{(-1)^{\ell}\ell! \ell^*!}{(2n-1)!}.
\end{equation}
Therefore
\begin{equation*}
  S_2=1+\frac{\delta_{\ell,0}+\delta_{\ell^*,0}}2
  -\beta_{\ell}  -\frac{1}{\ell+1}-\beta_{\ell^*}  -\frac{1}{\ell^*+1} -\frac{(-1)^{\ell}\ell! \ell^*!}{(2n-1)!}.
\end{equation*}
The sum $S_3$ is obtained from \e{recip2}, similarly to \e{sim}:
\begin{equation*}
  S_3= (-1)^\ell \beta_{\ell}\beta_{\ell^*}
  -\beta_{2n-1}\left(\frac{(-1)^{\ell}\ell! \ell^*!}{(2n-1)!}  +\frac{1}{\ell+1}+\frac{1}{\ell^*+1}\right).
\end{equation*}

The identity \e{try} that we are trying to prove becomes
\begin{multline} \label{big}
  -\frac{\beta_{2n-1}}{2n-2} \left[\binom{\ell}{1}+ \binom{\ell}{1^*}+\binom{\ell^*}{1}+ \binom{\ell^*}{1^*}\right]\\
  + \beta_{2n-1} \left[1+\frac{\delta_{\ell,0}+\delta_{\ell^*,0}}2
  -\beta_{\ell}  -\frac{1}{\ell+1}-\beta_{\ell^*}  -\frac{1}{\ell^*+1} -\frac{(-1)^{\ell}\ell! \ell^*!}{(2n-1)!} \right]  \\
  + (-1)^{\ell+1} \beta_{\ell}\beta_{\ell^*}
  +\beta_{2n-1}\left(\frac{(-1)^{\ell}\ell! \ell^*!}{(2n-1)!}  +\frac{1}{\ell+1}+\frac{1}{\ell^*+1}\right)
  \\
   = - \beta_{2n-1}\left(\beta_{\ell}+\beta_{\ell^*} + \frac{\delta_{\ell,1}+\delta_{\ell^*,1}}{2n-2}\right)+\beta_{\ell}\beta_{\ell^*},
\end{multline}
where the right side is replaced by $0$ if $\ell$ or $\ell^*$ equal $0$.
After cancelling, the left side of \e{big} reduces to
\begin{align*}
  -\frac{\beta_{2n-1}}{2n-2} \left[\binom{\ell}{1}+ \binom{\ell}{1^*}  +\binom{\ell^*}{1}+ \binom{\ell^*}{1^*}\right]&\\
  + \beta_{2n-1} &\left[1+\frac{\delta_{\ell,0}+\delta_{\ell^*,0}}2
  -\beta_{\ell}  -\beta_{\ell^*} \right]
   -(-1)^{\ell} \beta_{\ell}\beta_{\ell^*},
\end{align*}
and it may be verified that this is $0$ when $\ell=0$ or $\ell^*=0$, as desired. Cancelling further in \e{big} when $1\lqs \ell \lqs 1^*$ yields
\begin{equation*}
  \frac{1}{2n-2} \left[\binom{\ell}{1}+ \binom{\ell}{1^*}+\binom{\ell^*}{1}+ \binom{\ell^*}{1^*}\right]
  = 1+ \frac{\delta_{\ell,1}+\delta_{\ell^*,1}}{2n-2},
\end{equation*}
which is indeed true.
\end{proof}

\section{Expressing $\Psi_n$ in terms of $G_n$} \label{ddu}

Any hyperbolic conjugacy class in $\G/\{\pm I\}$ has a representative
\begin{equation} \label{tv}
  A=T^{n_1}V^{m_1} \cdots T^{n_r}V^{m_r},
\end{equation}
with all powers positive and a total of $p=n_1+m_1+ \cdots +n_r+m_r$ group elements.
Cyclic permutations of \e{tv} are made by removing  elements from the left and placing them on the right. We obtain $\tilde{A}_1=A$, $\tilde{A}_2=T^{-1}A T$, and so on. 
Also let $\tilde{c}_j$ and $\tilde{b}_j$  be the integers $c_{\tilde{A}_j}/\u(\tilde{A}_j)$ and $b_{\tilde{A}_j}/\u(\tilde{A}_j)$, respectively.

Using $V=TST$ in \e{tv} and replacing $A$ by $T^{-1}AT$ gives
\begin{equation} \label{tv2}
  A=T^{k_1} S T^{k_2} S \cdots S T^{k_q} S,
\end{equation}
with all powers at least $2$. Construct cyclic permutations of \e{tv2} slightly differently by removing blocks $T^{k_i} S$ from the left and placing them on the right.
This gives $A_1=A$, $A_2=(T^{k_1} S)^{-1} A T^{k_1} S$, etc.
Let $\omega_j$ and $\omega'_j$ be the fixed points of $A_{j}$. Also set $c_j$  to be $c_{A_j}/\u(A_j)$.
Note that since $\u$ is invariant under conjugation, \cite[Lemma 4]{duke}, we have $\u(A_j)= \u(\tilde{A}_j) = \u(A)$.

The key formula from Lemma 8 and Theorem 2 of \cite{duke} is established there using an interesting two-variable version of Eichler-Shimura cohomology. With the above notation and for $n\gqs 2$, it says
\begin{equation}\label{dkm}
  \Psi_n(A) = \ib_n \sum_{j=1}^{p} \left( (\tilde{c}_j)^{n-1} - (-\tilde{b}_j)^{n-1}\right) + \sum_{j=1}^{q}  c_j^{n-1} G_n(\omega_j, \omega_j').
\end{equation}
Duke  mentions in \cite[Sect. 2]{duke} that for each $n \gqs 2$, $\Psi_n(A)$ always being an integer is implied by the polynomial $G_n$ having integer coefficients. For completeness we prove this next.
\begin{prop} \label{ik}
For each  $n\gqs 2$, we have that $\Psi_n$ is integer valued if 
\begin{equation} \label{gnx}
   G_n(x) \in \Z[x_1,x_2].
\end{equation}
\end{prop}
\begin{proof}
Suppose that \e{gnx} is true. We may assume that $A$ is hyperbolic and, by \e{dkm}, we need only show that $c_j^{n-1} G_n(\omega_j, \omega_j')$ is always an integer. The symmetric polynomial $G_n(x_1,x_2)$ may be expressed as an element of $\Z[e_1,e_2]$ for the elementary symmetric polynomials $e_1=x_1+x_2$ and $e_2=x_1 x_2$ . To see this explicitly, note that $G_n(x_1,x_2)$ has terms of the form an integer times $(x_1 x_2)^m(x_1^\ell+x_2^\ell)$ for $m+\ell \lqs n-1$ and $\ell\gqs 1$. This follows from \e{qrx} since $G_n$ is a linear combination of the polynomials $Q_r$ for $r$ odd. Then
\begin{equation} \label{su}
  (x_1 x_2)^m(x_1^\ell+x_2^\ell) = \sum_{k=0}^{\ell/2} (-1)^k \frac{\ell}{\ell-k} \binom{\ell-k}{k}(x_1+x_2)^{\ell-2k}(x_1 x_2)^{m+k},
\end{equation}
by a classic result from symmetric functions \cite[Eq. (1)]{Gou}. The coefficients in \e{su} are integers since
\begin{equation*}
  \frac{\ell}{\ell-k} \binom{\ell-k}{k} = \binom{\ell-k}{k}+\binom{\ell-k-1}{k-1} \qquad (\ell>k>0).
\end{equation*}
 Write $a_j$, $b_j$ and $d_j$ for the other entries of $A_j$ divided by $\u(A_j)$. Along with $c_j$, these are all in $\Z$ and we have
\begin{equation*}
  \omega_j + \omega_j' = \frac{a_j-d_j}{c_j}, \qquad \omega_j \omega_j' = -\frac{b_j}{c_j}.
\end{equation*}
Hence,
\begin{equation*}
  (\omega_j \omega_j')^m(\omega_j^\ell+(\omega_j')^\ell) = \sum_{k=0}^{\ell/2} (-1)^k \frac{\ell}{\ell-k} \binom{\ell-k}{k}\left(\frac{a_j-d_j}{c_j}\right)^{\ell-2k}\left(-\frac{b_j}{c_j}\right)^{m+k}.
\end{equation*}
The largest possible power of $c_j$ in the denominator of the right side is $m+\ell$ and this is cancelled by the outside factor $c_j^{n-1}$.
\end{proof}

\begin{conj} \cite[Sect. 2]{duke} \label{gn-con}
For every $n\gqs 2$ we have $G_n(x) \in \Z[x_1,x_2]$.
\end{conj}

So we see that Conjecture \ref{gn-con} implies Conjecture \ref{con}. It seems that Conjecture \ref{gn-con} is strictly stronger since a polynomial may send $\Z$ to $\Z$ without having integer coefficients, (for example $\binom{x}{3}$).
Now, by \e{qrx}, \e{grx} we have  $G_n(x) \in \Z[x_1,x_2]$ if and only if
\begin{equation}\label{jk}
  \jb_n \binom{n-1}{u}  \binom{n-1}{r-u} \xi_r \in \Z \qquad \text{for all $u$, $r$ with $ 0\lqs u \lqs r\lqs 2n-2$}.
\end{equation}
As the expression in \e{jk} is unchanged as $r \to r^*$ we need only consider odd $r$ in the range $1\lqs r\lqs n-1$.
Substituting \e{xi}, \e{eee} and simplifying a little finds:

\begin{prop} \label{all}
For all $n\gqs 3$, $G_n(x) \in \Z[x_1,x_2]$ is equivalent to
\begin{equation} \label{nec}
  \jb_n \binom{n-1}{u}  \binom{n-1}{v}\left( \frac{B_{2n}}{2n}\left( \frac{B_{2w}}{2w}+\frac{B_{2n-2w}}{2n-2w}\right) - \frac{B_{2w}}{2w}\frac{B_{2n-2w}}{2n-2w} \right)
  +\frac{\delta_{w,1}}2 \in \Z,
\end{equation}
for all nonnegative integers  $u$, $v$, $w$ with $1\lqs w \lqs n/2$ and $u+v= 2w-1$.
\end{prop}

Studying \e{nec} numerically reveals that  the binomial coefficients are  needed to cancel denominators in a small proportion of cases.  For $r\in \Q_{\neq 0}$, let $\nu_p(r)$ be the usual $p$-adic valuation: $\nu_p(r)=m \in \Z$ when $r=p^m \alpha/\beta$ for $p\nmid \alpha,\beta$.

\begin{lemma} \label{bino}
Let  $u$, $v$, $w$ be  integers with $0\lqs u,v \lqs n-1$ and $u+v= 2w-1$ for  $1\lqs w \lqs n-1$. Then
\begin{align} \label{binh}
  \nu_p\left( \binom{n-1}{u}  \binom{n-1}{v}\right) &\gqs \begin{cases} \nu_p(2w)+\frac{2-(2n-2w)}{p-1}, \\
    \nu_p(2n-2w)+\frac{2-2w}{p-1}, 
  \end{cases}\\
\intertext{
for $p$  an odd prime, and}
  \nu_2\left( \binom{n-1}{u}  \binom{n-1}{v}\right) &\gqs \begin{cases} \nu_2(2w)-(n-w), \\
    \nu_2(2n-2w)-w.
  \end{cases} \label{bin2}
\end{align}
\end{lemma}
\begin{proof}
It is easy to verify \e{binh} and \e{bin2} when $w=1$, so we assume $w\gqs 2$. Without losing generality, suppose $u\gqs v$. Then $u\gqs w$ and it follows that $n-w$ divides the numerator $(n-1)(n-2) \cdots (n-u)$ of $\binom{n-1}{u}$, written with denominator $u!$. Hence, the left sides of \e{binh} and \e{bin2} are at least $\nu_p(n-w)-\nu_p(u!)-\nu_p(v!)$. The well-known formula
  $(p-1)\nu_p(m!)=m-s_p(m)$
has $s_p(m)$ equalling the sum of the base $p$ digits of $m$. Therefore
\begin{align*}
  \nu_p(n-w)-\nu_p(u!)-\nu_p(v!) &  = \nu_p(n-w)- \frac{u+v - s_p(u)-s_p(v)}{p-1} \\
   & \gqs \nu_p(n-w)- \frac{(2w-1) - 1}{p-1},
\end{align*}
proving the bottom inequality in \e{binh}.

For $p=2$ we may improve this slightly. Note that $s_2(u)+s_2(v)=1$ only for $v=0$ and $u$ a power of $2$. But if $v=0$ then $u=2w-1 \gqs 3$ is not a power of $2$. Therefore $s_2(u)+s_2(v)\gqs 2$. Also, considering the numerators of $\binom{n-1}{u}$ and $\binom{n-1}{v}$, excluding the $n-w$ factor, they must contain at least $\lfloor u/2\rfloor + \lfloor v/2\rfloor-1=w-2$ twos. Then the left side of  \e{bin2} is at least
$$
\nu_2(n-w) + (w-2)-(2w-1)+2= \nu_2(2(n-w))-w.
$$

Under the maps $u \mapsto n-1-u$ and $v \mapsto n-1-v$ we have $w \mapsto n-w$. This symmetry gives the top inequalities in \e{binh} and \e{bin2}.
\end{proof}

\section{Divisibility of Bernoulli numbers} \label{divb}

The Bernoulli number results we need in the next section are described here.
The following one is due to Clarke in \cite[Prop. 8]{cla}, explicitly determining  the fractional part of $B_n/n$.

\begin{theorem} \label{clar}
For $n$ even and positive,
\begin{equation}\label{bere}
  \frac{B_n}{n}=z_1(n)+\sum_{p-1\mid n} \frac{z_p(n)}{p^{1+\nu_p(n)}},
\end{equation}
where we are summing over  primes $p$. Here $z_1(n)$ and $z_p(n)$ are integers with $1\lqs z_p(n) < p^{1+\nu_p(n)}$. Precisely,
\begin{equation} \label{cas2}
  z_p(n) \equiv \left( 2\delta_{p,2} \delta_{n,2}+\left(\frac n{(p-1)p^{\nu_p(n)}}\right)^{-1} \right)\bmod p^{1+\nu_p(n)}.
\end{equation}
\end{theorem}

Our formulation in \e{bere} and \e{cas2} slightly simplifies \cite[Prop. 8]{cla} and includes the case $n=2$.  Since $\gcd(p,z_p(n))=1$, Von Staudt's second theorem is a corollary: for even $n\gqs 2$, the denominator of $B_n/n$ is
\begin{equation} \label{vs2}
  \jb_{n/2} = \prod_{p-1\mid n} p^{1+\nu_p(n)}.
\end{equation}

The next  result gives detailed information about differences of Bernoulli numbers and is contained in  \cite[Prop. 11.4.4]{Co}.
First set
\begin{equation*}
  H_p(n,m):= \begin{cases} \Bigl(1-\frac 1p\Bigr)\Bigl(\frac 1n-\frac 1m\Bigr)-(n-m)\delta_{p,3} & \text{if \ $p-1\mid n$},\\
 0 & \text{if \ $p-1\nmid n$}.
  \end{cases}
\end{equation*}

\begin{theorem} \label{kummer}
Let $n$ and $m$ be even and positive. If $p$ is  prime  and $n \equiv m \bmod (p-1)p^N$  then
\begin{equation}
  \left(1-p^{n-1} \right) \frac{B_n}{n} - \left(1-p^{m-1} \right) \frac{B_{m}}{m}  = H_p(n,m)+\frac{\alpha}{\beta} p^{N+1}, \label{kuu}
\end{equation}
for some rational $\alpha/\beta$ with $p \nmid \beta$.
\end{theorem}

The cases of this theorem when $p-1\nmid n$ are known as the classical Kummer congruences. They were extended by Cohen to $p-1\mid n$  using the theory of the Kubota-Leopoldt $p$-adic
zeta function in  Chapter 11 of \cite{Co}.

\section{Proof of Conjecture \ref{gn-con}} \label{pf}

\begin{theorem} \label{main}
Let $k$ and $\ell$ be positive and even with $k\lqs \ell$ and $m=k+\ell \gqs 6$.
Then for all nonnegative integers $u$, $v$ with $u+v=k-1$,
\begin{equation} \label{b3}
  \jb_{m/2}\binom{m/2-1}{u} \binom{m/2-1}{v}\left(\frac{B_m}{m}\left(\frac{B_\ell}{\ell} +\frac{B_k}{k}\right) - \frac{B_\ell}{\ell}\frac{B_k}{k}\right) +\frac{\delta_{k,2}}2 \in \Z.
\end{equation}
\end{theorem}
\begin{proof}
Let $P$ be the set of primes $p$ such that $p-1$ divides all three of $k$, $\ell$ and $m$. Then $\{2,3\} \subseteq P$. Let $K$, $L$ and $M$ be the primes with $p-1$ dividing only $k$, $\ell$ and $m$, respectively. (If $p-1$ divides any two of $k$, $\ell$ and $m$, then it must divide the third.) With Theorem \ref{clar}, write
\begin{align*}
   \frac{B_m}{m} & =a_1+\sum_{p\in P} \frac{a_p}{p^{1+\nu_p(m)}} + \sum_{p\in M} \frac{a_p}{p^{1+\nu_p(m)}},\\
  \frac{B_\ell}{\ell} & =b_1+\sum_{p\in P} \frac{b_p}{p^{1+\nu_p(\ell)}} + \sum_{p\in L} \frac{b_p}{p^{1+\nu_p(\ell)}},\\
  \frac{B_k}{k} & =c_1+\sum_{p\in P} \frac{c_p}{p^{1+\nu_p(k)}} + \sum_{p\in K} \frac{c_p}{p^{1+\nu_p(k)}}.
\end{align*}
Our goal is to show that  the left side of \e{b3} is $p$-integral, (has $p$-adic valuation at least $0$), for all  $p$ in $P$, $M$, $L$ and $K$. Note that, by \e{vs2},
\begin{equation*}
   \jb_{m/2} = \prod_{p \in P\, \cup \,M} p^{1+\nu_p(m)}.
\end{equation*}

{\bf Primes in $\bm{M}$.} The case $p\in M$ is easily dealt with: the $p^{1+\nu_p(m)}$ factor from $\jb_{m/2}$ cancels the same factor only appearing in the denominator of $B_m/m$. So   the left side of \e{b3} is $p$-integral.

{\bf Primes in $\bm{L}$, $\bm{K}$.} Next take $p \in L$.
Express the Bernoulli number part of \e{b3} as
\begin{equation} \label{b3b}
  \frac{B_\ell}{\ell}\left(\frac{B_m}{m} -\frac{B_k}{k}\right) + \frac{B_m}{m}\frac{B_k}{k}.
\end{equation}
 Then $p-1$ does not divide $m$ and $k$, while $p-1$ does divide $m-k=\ell$ with
\begin{equation*}
  m \equiv k \bmod (p-1)p^{\nu_p(\ell)}.
\end{equation*}
By Theorem \ref{kummer},
\begin{equation} \label{eb}
  \left(1-p^{m-1} \right) \frac{B_m}{m} - \left(1-p^{k-1} \right) \frac{B_k}{k} = H_p(m,k)+\frac \alpha\beta p^{1+\nu_p(\ell)},
\end{equation}
for rational $\alpha/\beta$ with $p \nmid \beta$ and $H_p(m,k)=0$. Hence
\begin{equation} \label{io}
 \frac{1}{p^{1+\nu_p(\ell)}}\left(\frac{B_m}{m} -\frac{B_k}{k}\right) = \frac \alpha\beta
 + p^{m-2-\nu_p(\ell)} \frac{B_m}{m} -p^{k-2-\nu_p(\ell)}\frac{B_k}{k}.
\end{equation}
If $k-2-\nu_p(\ell)<0$ then multiplication by the binomial coefficients ensures that the left side of \e{b3} is $p$-integral, since
\begin{equation} \label{vpb}
   \nu_p\left( \binom{m/2-1}{u}  \binom{m/2-1}{v}\right) \gqs \nu_p(\ell)+\frac{2-k}{p-1} \gqs \nu_p(\ell)+2-k,
\end{equation}
by \e{binh} in Lemma \ref{bino}. The case $p\in K$ is handled in the same way, using  the rearrangement
\begin{equation} \label{b4b}
  \frac{B_k}{k}\left(\frac{B_m}{m} -\frac{B_\ell}{\ell}\right) + \frac{B_m}{m}\frac{B_\ell}{\ell}.
\end{equation}

{\bf Primes in $\bm{P-\{2\}}$ with $\bm{\nu_p(m)}$ small.} That leaves the primes in $P$ to check. First take $p$ in the case where $N:=\nu_p(\ell)$ is strictly larger than $\nu_p(m)$ and $\nu_p(k)$. Then necessarily $\nu_p(m)=\nu_p(k)$ and let $x$ be this common valuation.
Assume that $p \neq 2$.
The only terms of \e{b3b} that can possibly  be non $p$-integral after multiplication by $\jb_{m/2}$ are
\begin{equation*}
  \frac{b_p}{p^{N+1}}\left(\frac{B_m}{m} -\frac{B_k}{k}\right) + \frac{a_p}{p^{x+1}}\frac{c_p}{p^{x+1}}.
\end{equation*}
Here $\jb_{m/2}$ contains a $p^{x+1}$ factor, so with Theorem \ref{kummer} and \e{eb} again, we wish to show a nonnegative valuation for
\begin{equation} \label{sun}
  p^{x+1}\left(\frac{b_p}{p^{N+1}}H_p(m,k) + b_p\frac \alpha\beta
 + b_p p^{m-2-N} \frac{B_m}{m} -b_p p^{k-2-N}\frac{B_k}{k}  + \frac{a_p c_p}{p^{2x+2}} \right)
\end{equation}
multiplied by the binomial coefficients. The terms containing $B_m/m$ and $B_k/k$ have total valuation at least $k-2-N$. If $k-2-N<0$ then  the binomial coefficients, with valuation at least
\begin{equation*}
  \nu_p(\ell)+\frac{2-k}{p-1} \gqs N+1-k/2
\end{equation*}
by \e{binh}, ensure  $p$-integrality.
Write the integers
\begin{equation*}
   m'=\frac{m}{(p-1) p^{x}}, \qquad k'=\frac{k}{(p-1) p^{x}}, \qquad \ell'=\frac{\ell}{(p-1) p^{N}},
\end{equation*}
so that
\begin{equation*}
  H_p(m,k)=-\frac{\ell'}{m'k'}p^{N-1-2x}- \ell'(p-1)p^N \delta_{p,3}.
\end{equation*}
Then \e{sun} without the $\alpha/\beta$, $B_m/m$ and $B_k/k$  terms equals
\begin{equation*}
  -b_p \ell'(p-1)p^x \delta_{p,3} + \frac{a_p c_p}{p^{x+1}} -\frac{ b_p  \ell'}{m'k' p^{x+1}}
\end{equation*}
and it remains to  show $\nu_p\left(m'k' a_p c_p - b_p  \ell'\right) \gqs x+1$.
But by \e{cas2} we know $m' a_p \equiv k' c_p \equiv \ell' b_p \equiv 1 \bmod p^{x+1}$, and hence
\begin{equation*}
   m' a_p \cdot k' c_p - \ell' b_p
 \equiv 1 \cdot 1 -  1  \equiv 0 \bmod p^{x+1},
\end{equation*}
as we wanted.

{\bf The case $\bm{2\in P}$ with $\bm{\nu_2(m)}$ small.}
We continue the above analysis for $p=2$, and first suppose that $k\neq 2$. The
arguments go through as before, with the only difference being that the binomial coefficients now just have valuation at least $N-k/2$ by \e{bin2}. This ensures  $p$-integrality of the terms in \e{sun} containing $B_m/m$ and $B_k/k$  for $k\gqs 4$.

If $k=2$ then $x=1$ and we must deal with the $B_k/k$ term in \e{sun} as well as the extra terms  $\delta_{k,2}/2$ from \e{b3} and $a_2/2$ from the $2\delta_{p,2} \delta_{k,2}$ part of $c_p$ in \e{cas2}. Altogether, including the binomial coefficients factor, this is
\begin{align*}
  \left(\frac m2-1\right)\left(   \frac{a_2}2 - \frac{b_2}{3\cdot 2^{N}}\right) +\frac 12 & =
  \ell' 2^{N-1}\frac{a_2}2 -\ell' 2^{N-1} \frac{b_2}{3 \cdot 2^{N}}  +\frac 12 \\
   & = \ell' 2^{N-2}a_2 + \frac{3-\ell' b_2}6.
\end{align*}
This is $2$-integral since $N>x=1$ and $\ell' b_2 \equiv 1 \bmod 4$.

It follows that   the left side of \e{b3} is $p$-integral for all $p \in P$ where $\nu_p(\ell)$ is strictly greater than $\nu_p(m)$ and $\nu_p(k)$. Clearly the same argument works if $\nu_p(k)$ is strictly greater than $\nu_p(m)$ and $\nu_p(\ell)$.

{\bf Primes in $\bm{P}$ with $\bm{\nu_p(m)}$ large.} The final primes to check are $p \in P$ with  $\nu_p(\ell)=\nu_p(k)=x$ and $\nu_p(m)=N \gqs x$. Assume that $p$ and $k$ are not both $2$. Now $\jb_{m/2}$ contains the factor $p^{N+1}$, and the only terms of \e{b3} that can be non $p$-integral  are
\begin{equation} \label{bd}
   p^{N+1}\left(\frac{a_p}{p^{N+1}}\frac{ b_p}{p^{x+1}}
  + \frac{a_p}{p^{N+1}}\frac{c_p}{p^{x+1}}- \frac{ b_p}{p^{x+1}}\frac{c_p}{p^{x+1}}\right),
\end{equation}
omitting the binomial coefficient factors. This time we need the integers
\begin{equation*}
  k'=\frac{k}{(p-1)p^x}, \qquad \ell'=\frac{\ell}{(p-1)p^x}, \qquad m'=\frac{m}{(p-1)p^{N}}.
\end{equation*}
Then $m' p^{N-x} = k'+\ell'$ and $m'a_p \equiv \ell'b_p \equiv k' c_p \equiv 1 \bmod p^{x+1}$.
 We claim
that $p^{x+1}$ divides the numerator of
\begin{equation} \label{jj}
  \frac{a_p b_p  +a_p c_p- b_p c_p p^{N-x}}{p^{x+1}}
\end{equation}
from \e{bd}. Evidently,
\begin{align*}
  m'\ell'k'\left( a_p b_p  +a_p c_p- b_p c_p p^{N-x}\right) & =  m' a_p  \cdot \ell' b_p \cdot k' +m'a_p  \cdot k' c_p \cdot \ell'- \ell' b_p  \cdot k' c_p  \cdot m' p^{N-x}\\
   & \equiv 1  \cdot 1 \cdot k' +1  \cdot 1 \cdot \ell'- 1  \cdot 1  \cdot (k'+\ell') \equiv 0 \bmod p^{x+1}.
\end{align*}
This implies our claim, since $m'\ell'k'$ is relatively prime to $p$, and so \e{bd} is $p$-integral, as desired.

{\bf The case $\bm{2\in P}$, $\bm{k=2}$ with $\bm{\nu_2(m)}$ large.} The last case  of $p=2$ and $k= 2$    has some further terms to examine. Here $x=1$ and $\ell>2$. The extra possibly non $2$-integral piece of \e{b3}, corresponding to the $2\delta_{p,2} \delta_{k,2}$ part of $c_p$ in \e{cas2}, and  the $\delta_{k,2}/2$ term, is
\begin{equation} \label{apr}
  \left(\frac m2-1\right)\frac{   a_2 -  b_2  \cdot 2^{N-x}}{2} +\frac 12,
\end{equation}
for $x=1$. Since $m/2-1=\ell/2$ is odd, then \e{apr} is an integer if we can show that $a_2 -  b_2  \cdot 2^{N-x}$ is odd. To verify this, use the congruences before \e{jj} to show
\begin{align*}
  m'\ell'(a_2 -  b_2  \cdot 2^{N-x}) & = m' a_2 \cdot \ell' - \ell'b_2 \cdot m'2^{N-x}  \\
  & \equiv 1 \cdot \ell' - 1 \cdot (k'+\ell') \bmod 4 \\
&  \equiv - k' \bmod 4,
\end{align*}
for the odd numbers $m'$, $\ell'$ and $k'$.
This completes the proof of Theorem \ref{main}.
\end{proof}

\begin{cor} \label{cos}
Conjectures \ref{con} and \ref{gn-con} are true.
\end{cor}
\begin{proof}
Theorem \ref{main} and Proposition \ref{all} imply Conjecture \ref{gn-con}. This in turn implies Conjecture \ref{con} via Proposition \ref{ik}.
\end{proof}

\section{Further formulas for $\Psi_n$} \label{mor}

For $A=\left(\begin{smallmatrix}a&b\\c&d\end{smallmatrix}\right)$ in  $\G$, set $j_A(z):=cz+d$ and  define the  quadratic form
\begin{equation*}
  q_A(z):=(Az-z)j_A(z) = c z^2+(d-a)z-b.
\end{equation*}
Then
\begin{equation*}
  q_A( B z)  \cdot j_B(z)^2 =q_{B^{-1} AB}(z)
\end{equation*}
and hence
\begin{equation} \label{qry}
  q_A(z+k)  =q_{T^{-k} A T^k}(z), \qquad q_A(-1/z) \cdot z^2  =q_{S^{-1} A S}(z).
\end{equation}
Next, for fixed $n$, set
\begin{equation} \label{fraa}
  \qq_r(A):=[z^r] q_A(z)^{n-1},
\end{equation}
so that, for example,
\begin{align}\label{eg}
  \qq_0(A)=(-b)^{n-1}, & \qquad \qq_1(A)=(n-1)(-b)^{n-2}(d-a), \\
  \qq_{0^*}(A)=c^{n-1}, & \qquad \qq_{1^*}(A)=(n-1)c^{n-2}(d-a).
\end{align}
Following from \e{qry},
\begin{equation} \label{qry2}
  \qq_r(T^{-k} A T^k) =\sum_{m=0}^{2n-2} \binom{m}{r} k^{m-r} \qq_m(A), \qquad \qq_r(S^{-1} A S) = (-1)^r \qq_{r^*}(A).
\end{equation}
Set
\begin{equation*}
  \Psi_n^{(0)}(A) := \u(A)^{1-n}\sum_{j=1}^p \left(\qq_{0^*}(\tilde{A}_j) - \qq_{0}(\tilde{A}_j) \right) =
  \sum_{j=1}^{p} \left( (\tilde{c}_j)^{n-1} - (-\tilde{b}_j)^{n-1}\right),
\end{equation*}
as in the first part of \e{dkm} and using its notation. The next result is stated in \cite[Eq. (2.21)]{duke} and we give a direct proof.

\begin{theorem} \label{zero}
We have
\begin{equation}\label{hv}
  \Psi_n^{(0)}(A) = \u(A)^{1-n}\sum_{j=1}^{q} \sum_{r=0}^{2n-2}\left(  \frac{B_{r+1}}{r+1}+\frac{B_{r^*+1}}{r^*+1}
  -\frac{k_j^{r+1}}{r+1}
  + \delta(r)
  \right) \qq_r(A_j),
\end{equation}
for $\delta(r)= (\delta_{r,0}+\delta_{r^*,0})/2 - (\delta_{r,1}+\delta_{r^*,1})/(2n-2)$.
\end{theorem}
\begin{proof}
To translate between \e{tv} and \e{tv2}, rewrite \e{tv} as
\begin{equation*}
  A=T^{k_1-2}V T^{k_2-2}V \cdots V T^{k_q-2}V
\end{equation*}
for $k_i\gqs 2$. Then $T A T^{-1}$ equals the right side of \e{tv2} and we see
\begin{equation*}
   \tilde{A}_1= T^{-1} A_1 T, \qquad \tilde{A}_{(k_1-1)+1}= T^{-1} A_2 T, \qquad \tilde{A}_{(k_1-1)+(k_2-1)+1}= T^{-1} A_3 T,
\end{equation*}
and so on. Set $s_0:=0$ and $s_j:=(k_1-1)+(k_2-1)+ \cdots +(k_{j}-1)$. Then
\begin{equation*}
  \tilde{A}_{s_{j-1}+u} = T^{-u} A_j T^{u} \qquad \text{for} \quad 1\lqs j \lqs q, \quad 1\lqs u \lqs k_j-1.
\end{equation*}
Hence
\begin{align*}
  \sum_{j=1}^p \left(\qq_{0^*}(\tilde{A}_j) - \qq_{0}(\tilde{A}_j) \right)  & = \sum_{j=1}^q \sum_{u=1}^{k_j-1} \left(\qq_{0^*}(T^{-u} A_j T^{u}) - \qq_{0}(T^{-u} A_j T^{u}) \right)\\
   & = \sum_{j=1}^q \sum_{u=1}^{k_j-1} \left(\qq_{0^*}( A_j ) - \sum_{r=0}^{2n-2}  u^{r} \qq_r(A_j) \right),
\end{align*}
employing the left side of \e{qry2}. Let
\begin{equation*}
  S_r(k):= 1^r+2^r+\cdots +k^r = \frac 1{r+1} \sum_{\ell=0}^{r} \binom{r+1}{\ell}B_\ell \cdot (k+1)^{r+1-\ell}-\delta_{r,0},
\end{equation*}
using our last required property of the Bernoulli numbers, as in \cite[Eq. (2.2)]{AD08}.
Then
\begin{equation*}
  \sum_{j=1}^p \left(\qq_{0^*}(\tilde{A}_j) - \qq_{0}(\tilde{A}_j) \right)
  = \sum_{j=1}^q  \left((k_j-1)\qq_{0^*}( A_j ) - \sum_{r=0}^{2n-2}  S_r(k_j-1) \qq_r(A_j) \right).
\end{equation*}
Focussing on the sum over $r$ finds
\begin{multline*}
  - \sum_{r=0}^{2n-2}  S_r(k_j-1) \qq_r(A_j)  = \qq_0(A_j)- \sum_{r=0}^{2n-2}  \frac 1{r+1} \sum_{\ell=0}^{r} \binom{r+1}{\ell}B_\ell k_j^{r+1-\ell} \qq_r(A_j) \\
    = \qq_0(A_j) - \sum_{r=0}^{2n-2} \frac{k_j^{r+1}}{r+1} \qq_r(A_j)- \sum_{r=0}^{2n-2}  \frac 1{r+1} \sum_{\ell=1}^{r} \binom{r+1}{\ell}B_\ell k_j^{r+1-\ell} \qq_r(A_j).
\end{multline*}
Write the third term above as
\begin{equation*}
  -\sum_{\ell=0}^{2n-2} \frac{B_{\ell+1}}{\ell+1} \sum_{r=\ell+1}^{2n-2} \binom{r}{\ell} k_j^{r-\ell} \qq_r(A_j)
  = \sum_{\ell=0}^{2n-2} \frac{B_{\ell+1}}{\ell+1}\qq_\ell(A_j) -
  \sum_{\ell=0}^{2n-2} \frac{B_{\ell+1}}{\ell+1} \sum_{r=\ell}^{2n-2} \binom{r}{\ell} k_j^{r-\ell} \qq_r(A_j).
\end{equation*}
With \e{qry2}, the last term is recognized as
\begin{align*}
  -\sum_{\ell=0}^{2n-2} \frac{B_{\ell+1}}{\ell+1} \sum_{r=\ell}^{2n-2} \binom{r}{\ell} k_j^{r-\ell} \qq_r(A_j) &  = -\sum_{\ell=0}^{2n-2} \frac{B_{\ell+1}}{\ell+1}  \qq_\ell(T^{-k_j} A_j T^{k_j}) \\
   & = -\sum_{\ell=0}^{2n-2} \frac{B_{\ell^*+1}}{\ell^*+1}  \qq_{\ell^*}(T^{-k_j} A_j T^{k_j}) \\
   & = -\sum_{\ell=0}^{2n-2} \frac{B_{\ell^*+1}}{\ell^*+1}  (-1)^\ell \qq_{\ell}(S^{-1}T^{-k_j} A_j T^{k_j}S) \\
   & = -\sum_{\ell=0}^{2n-2} \frac{B_{\ell^*+1}}{\ell^*+1}  (-1)^\ell \qq_{\ell}(A_{j+1})\\
   & = \qq_{0^*}(A_{j+1})+\sum_{\ell=0}^{2n-2} \frac{B_{\ell^*+1}}{\ell^*+1}   \qq_{\ell}(A_{j+1}).
\end{align*}
Altogether we have shown
\begin{multline*}
  \sum_{j=1}^p \left(\qq_{0^*}(\tilde{A}_j) - \qq_{0}(\tilde{A}_j) \right)
  = \sum_{j=1}^q  \left( (k_j-1)\qq_{0^*}( A_j ) +\qq_0(A_j) - \sum_{r=0}^{2n-2} \frac{k_j^{r+1}}{r+1} \qq_r(A_j)\right. \\
  \left. + \sum_{\ell=0}^{2n-2} \frac{B_{\ell+1}}{\ell+1}\qq_\ell(A_j) + \qq_{0^*}(A_{j+1})+\sum_{\ell=0}^{2n-2} \frac{B_{\ell^*+1}}{\ell^*+1}   \qq_{\ell}(A_{j+1})\right).
\end{multline*}
Since $A_{q+1}=A_1$ for these cyclic permutations, this is
\begin{equation*}
  \sum_{j=1}^q  \left( k_j \qq_{0^*}( A_j ) +\qq_0(A_j)\right)+
  \sum_{j=1}^{q} \sum_{r=0}^{2n-2}\left(  \frac{B_{r+1}}{r+1}+\frac{B_{r^*+1}}{r^*+1}
  -\frac{k_j^{r+1}}{r+1}
  \right) \qq_r(A_j).
\end{equation*}
To get the first sum above into the final form, note first that
\begin{align*}
  \qq_r(A_{j+1}) & = \qq_r(S^{-1}T^{-k_j}A_j T^{k_j}S) \\
   & = (-1)^r \qq_{r^*}(T^{-k_j}A_j T^{k_j})\\
   & = (-1)^r \sum_{m=0}^{2n-2} \binom{m}{r^*} k_j^{m-r^*}\qq_{m}(A_j),
\end{align*}
by \e{qry2}. In the cases $r=0$, $1$ we obtain
\begin{equation*}
  \qq_0(A_{j+1})= \qq_{0^*}(A_j), \qquad (2n-2)k_j \qq_{0^*}(A_j) = -\qq_1(A_{j+1})-\qq_{1^*}(A_j).
\end{equation*}
Hence
\begin{equation*}
  \sum_{j=1}^q  \left( k_j \qq_{0^*}( A_j ) +\qq_0(A_j)\right)
  = \sum_{j=1}^q  \left( \frac 12 \left(\qq_0(A_j)+ \qq_{0^*}( A_j )\right) - \frac 1{2n-2} \left(\qq_1(A_j)+ \qq_{1^*}( A_j )\right)\right),
\end{equation*}
and \e{hv} follows.
\end{proof}

More generally, the same reasoning gives formulas for the integer-valued spanning symbols
\begin{equation*}
  \Psi_n^{(m)}(A) := \u(A)^{1-n}\sum_{j=1}^p \left(\qq_{m^*}(\tilde{A}_j) - (-1)^m\qq_{m}(\tilde{A}_j) \right),
\end{equation*}
from \cite[Sect. 6]{duke}. The higher Rademacher symbol $\Psi_n$ may be expressed as linear combinations of these, though not uniquely.

\begin{theorem} \label{zero2}
Suppose $n\gqs 2$ and $0\lqs m \lqs 2n-2$. For hyperbolic $A$, in the notation of \e{eee}, \e{dkm},
\begin{multline}\label{hvx}
  \Psi_n^{(m)}(A) = \u(A)^{1-n}\sum_{j=1}^{q} \sum_{r=0}^{2n-2}\left[
\binom{r}{m^*}\beta_{r-m^*}-\binom{r^*}{m}\beta_{r^*-m}
-(-1)^m \binom{r}{m}\beta_{r-m}
\right.\\
\left.
+(-1)^m\binom{r^*}{m^*}\beta_{r^*-m^*}
  -(-1)^m\binom{r}{m}\frac{k_j^{r-m+1}}{r-m+1}
  + \binom{r}{m^*}\frac{k_j^{r-m^*+1}}{r-m^*+1}
  \right] \qq_r(A_j).
\end{multline}
\end{theorem}

If $A$ fixes $\omega$ and $\omega'$ then, from a short calculation,
\begin{align*}
  c^{n-1} Q_r(\omega,\omega') & = \qq_r(A), \\
  c^{n-1} G_n(\omega,\omega') & = \jb_n \sum_r \xi_r \cdot \qq_r(A).
\end{align*}
So formula \e{dkm} may be rephrased with the $\qq_r$ notation \e{fraa} as
\begin{equation} \label{frp}
  \Psi_n(A) = \ib_n \Psi_n^{(0)}(A)+ \u(A)^{1-n} \jb_n\sum_{j=1}^{q} \sum_{r=0}^{2n-2} \xi_r \cdot \qq_r(A).
\end{equation}
Inserting the expressions from Theorem \ref{zero} and \e{xi} into \e{frp} finds a lot of cancellation, giving the following compact result.
\begin{cor} \label{coo}
For $n\gqs 2$ and hyperbolic $A \in \G$ we have
\begin{equation}\label{hv2}
  \Psi_n(A) = \u(A)^{1-n} \jb_n\sum_{j=1}^{q} \sum_{r=1}^{2n-1}\left( \frac{B_{2n}}{2n}\frac{k_j^{r}}{r}
  -
   \frac{B_{r}}{r}
  \frac{B_{2n-r}}{2n-r}
  \right) \qq_{r-1}(A_j).
\end{equation}
\end{cor}
This is  also  discussed in \cite[Sect. 3]{duke}.
The left side of \e{hv2} may be replaced with $\jb_n \zeta(1-n, \mathcal A)$ by \e{zet}. In this way we have established another proof of Zagier's formula in \cite[p. 149]{Z76}. Similar formulas are shown in \cite{VZ13} and as the main theorem of \cite{JL16}.


{\small \bibliography{duke-biblio} }

{\small 
\vskip 5mm
\noindent
\textsc{Department of Mathematics, The CUNY Graduate Center, 365 Fifth Avenue, New York, NY 10016-4309, U.S.A.}

\noindent
{\em E-mail address:} \texttt{cosullivan@gc.cuny.edu}
}

\end{document}